\documentclass[11pt]{article}

\usepackage[top=1in, bottom=1.25in, left=0.75in, right=0.75in]{geometry}
\usepackage{hyperref}
\usepackage{amssymb}
\usepackage[linesnumbered,lined,ruled]{algorithm2e}
\usepackage{amsmath,amsthm,amsfonts,dsfont,mathtools,bm}
\usepackage{url}
\usepackage{rotating}
\usepackage{multirow}
\usepackage{color}
\usepackage{xcolor}
\usepackage{enumitem}
\DeclareSymbolFont{rsfs}{U}{rsfs}{m}{n}
\DeclareSymbolFontAlphabet{\mathscrsfs}{rsfs}

\allowdisplaybreaks
\newtheorem{theorem}{Theorem}

\newtheorem{lemma}[theorem]{Lemma}

\newtheorem{remark}{Remark}

\newtheorem{proposition}[theorem]{Proposition}

\renewcommand{\P}{\operatorname{\mathbb{P}}}

\newcommand{\E}{\operatorname{\mathbb{E}}}

\newcommand{\R}{\mathbb{R}}

\newcommand{\rmd}{\mathrm{d}}

\DeclareMathOperator{\cov}{\mathrm{cov}}
\newcommand{\Tr}{\mathrm{Tr}}

\newcommand{\sTAP}{\mbox{\rm\tiny TAP}}
\newcommand{\eps}{\varepsilon}

\def\b0{{\boldsymbol{0}}}
\def\normal{{\sf N}}

\def\cE{{\mathcal E}}

\def\cuF{\mathscrsfs{F}}

\def\<{{\langle}}
\def\>{{\rangle}}

\numberwithin{equation}{section}
\numberwithin{theorem}{section}

\begin{document}

\title{Bounds on the covariance matrix of the Sherrington--Kirkpatrick model}

\author{Ahmed El Alaoui\thanks{Department of Statistics and Data Science, Cornell University. Email: elalaoui@cornell.edu.}, 
\;\; Jason Gaitonde\thanks{Department of Computer Science, Cornell University.}}

\date{}
\maketitle

\begin{abstract}
We consider the Sherrington-Kirkpatrick model with no external field and inverse temperature $\beta<1$ and prove that the expected operator norm of the covariance matrix of the Gibbs measure is bounded by a constant depending only on $\beta$.  
This answers an open question raised by Talagrand, who proved a bound of $C(\beta) (\log n)^8$. 
Our result follows by establishing an approximate formula for the covariance matrix which we obtain by differentiating the TAP equations and then optimally controlling the associated error terms. We complement this result by showing diverging lower bounds on the operator norm, both at the critical and low temperatures.       
\end{abstract}


\section{Introduction and main result}
\label{sec:1}
We consider the Sherrington--Kirkpatrick model of spin glasses, a Gibbs distribution over the hypercube $\{-1,+1\}^n$ given by the expression 
\begin{equation}\label{eq:sk}
\mu(\sigma) = \, \frac{1}{Z} \exp\Big\{ \frac{\beta}{\sqrt{n}}\sum_{ i<j } g_{ij} \sigma_i \sigma_j \Big\}\, , ~~~~~  \sigma \in \{-1,+1\}^n \, , 
\end{equation}
where $\beta \ge  0$ is the inverse temperature parameter and $(g_{ij} )_{i<j}$ are the random coupling coefficients assumed to be i.i.d.\ $\normal(0,1)$ random variables. 
We are interested in the behavior of the $n \times n$ covariance matrix of this probability measure: 
\begin{equation}\label{eq:cov}
\big(\cov(\mu)\big)_{ij} = \big\langle \sigma_i \,  \sigma_j \big\rangle  \, , ~~~~~ i , j \in [1,n] \, .
\end{equation}
Here, the brackets indicate the average with respect to $\mu$. It is expected that the operator norm of $\cov(\mu)$ is of constant order in $n$ whenever the model is at high temperature, i.e., for all $\beta <1$, and that it must diverge at the critical and low temperatures $\beta \ge 1$. Talagrand proved that for all $\beta<1$ there exists $C(\beta)<\infty$ such that 
\[\E \|\cov(\mu)\|_{\textup{op}} \le C(\beta)(\log n)^{8}\, ,\] 
and conjectured that the logarithmic term can be removed entirely~\cite[Section 11.5]{talagrand2011mean2}. His proof relies on the moment method and bounds the expectation of the trace of large powers of the covariance matrix; a method known to be loose by a logarithmic factor for random matrices with i.i.d.\ entries.  

Recently, Bauerschmidt and Bodineau~\cite{bauerschmidt2019very} proved a decomposition theorem for Ising measures into a log-concave mixture of product measures, provided that their interaction matrix $J$ is positive semi-definite and satisfies the operator norm bound $\|J\|_{\text{op}} < 1$. The authors used this decomposition to prove a log-Sobolev inequality for such measure for a notion of the discrete gradient. 
In the special case of the SK model, their decomposition implies that $\E \|\cov(\mu)\|_{\text{op}}$ is bounded for all $\beta < 1/4$.  
See also the work of Eldan, Koehler and Zeitouni \cite{eldan2021spectral} who proved under the same conditions a spectral gap inequality for Glauber dynamics. Such functional inequalities are expected to hold in the entire high-temperature regime, this is however still an open problem.      
In this paper we show boundedness of the operator norm for all $\beta<1$:

\begin{theorem}\label{thm:main1}     
For $\beta<1$ there exists $C(\beta)<\infty$ such that $\E \|\cov(\mu)\|_{\textup{op}} \le C(\beta)$ for all $n \ge 1$. 
\end{theorem}

Extensions to a non-zero external field, possibly relevant to the problem of proving Poincaré and log-Sobolev inequalities for $\mu$ for all $\beta<1$, are discussed in Section~\ref{sec:2}. 
    
Conversely, we show diverging lower bounds on the operator norm for $\beta \to 1^-$, $\beta=1$ and $\beta>1$:

\begin{theorem}\label{thm:main3}
\phantom{AA}
\begin{itemize}
\item For $\beta<1$,  $\liminf_{n\to \infty} \E \|\cov(\mu)\|_{\textup{op}} \geq 1/(1-\beta^2)$. 
\item For $\beta=1$, $\E \|\cov(\mu)\|_{\textup{op}} \ge c\big(n \big / \log n\big)^{3/16}$,
where $c>0$ is an absolute constant, for all $n \ge 2$.
\item For $\beta >1$, there exists $c(\beta)>0$ such that  $\E \|\cov(\mu)\|_{\textup{op}} \ge c(\beta) n$ for all $n \ge 1$. 
\end{itemize}
\end{theorem}
 
In Proposition~\ref{prop:near_crit} we also prove a similar lower bound to the second bullet, albeit not at $\beta=1$ but at a temperature approaching criticality: $\beta = \beta_n \to 1^-$ as $n\to \infty$ such that $n^{1/3}(1-\beta_n^2) \to + \infty$. 
The arguments used to prove Theorem~\ref{thm:main3} follow more or less directly from known results, so we delay their exposition to Section~\ref{sec:lowerb} in favor of the high-temperature result.
  
Our approach to showing Theorem~\ref{thm:main1} is by first establishing a TAP equation for the two-point correlations $\langle \sigma_i \,  \sigma_j \rangle$ with an optimal error bound of $1/n^2$ as we write next. 

\begin{theorem}\label{thm:main2}     
Let $A \in \R^{n \times n}$ such that $A_{ii}=0$ and $A_{ij}= A_{ji}= g_{ij}/\sqrt{n}$ for $i<j$.
For $\beta<1$ there exists $C(\beta)<\infty$ such that 
\begin{equation}\label{eq:main1}
 \E \Big[ \Big\| \big((1+\beta^2)I - \beta A\big)\cov(\mu) - I\Big\|_{F}^2\Big] \le C(\beta) \, .
 \end{equation}
In particular we have for $i \neq j$,
\begin{equation}\label{eq:main2}
\E \Big[ \Big( (1+\beta^2) \langle \sigma_i \,  \sigma_j \rangle - \frac{\beta}{\sqrt{n}}\sum_{k=1}^n g_{ik}\langle \sigma_k \,  \sigma_j \rangle \Big)^2\Big] \le \frac{C(\beta)}{n^2} \, .
 \end{equation}
\end{theorem}
An analogue of Eq.~\eqref{eq:main2} with a weaker error bound was recently proved by Adhikari, Brennecke, von Soosten and Yau~\cite{adhikari2021dynamical} for  fixed external fields parallel to the all-ones vector $\mathbf{1}$ up to an inverse temperature $\beta_0 = \log 2$ using what the authors call the dynamical approach to the TAP equations. The error bound they prove is of the form $C(\beta,\eps) / n^{1+\eps}$ for $\eps$ sufficiently small. This bound is unfortunately not enough to imply Theorem~\ref{thm:main1}; a bound of order $1/n^2$ seems necessary for the conclusion to follow. We elaborate on a generalization of this result to nonzero external field in Section~\ref{sec:2} below.

Theorem~\ref{thm:main1} then follows by known bounds on the tail probability of the largest eigenvalue of a GOE matrix:  
\begin{proof}[Proof of Theorem~\ref{thm:main1}]
For $\beta<1$, let $\eps \in (0, (1-\beta)^2/(1+\beta))$ and define the event 
\[\cE= \big\{\lambda_{\min}\big((1+\beta^2)I-\beta A\big)\leq \varepsilon\big\} \, .\] 
Observe that on $\cE$ we have $\lambda_{\max} (A) \geq \frac{1+\beta^2-\varepsilon}{\beta}\geq 2+\varepsilon$ 
where the second inequality follows from the bound imposed on $\eps$. It follows that $\P(\cE)\leq \exp(-c(\eps) n)$, $c(\eps)>0$ by standard results on the largest eigenvalue of a  GOE matrix (see e.g.,~\cite[Exercise 7.3.5]{vershynin2018high}).
Now conditional on the complement event $\cE^c = \{\lambda_{\min}\big((1+\beta^2)I-\beta A\big) > \varepsilon\}$ we have 
\begin{align}
\|\cov(\mu)\|_{\text{op}} &\le \big\|\big((1+\beta^2)I-\beta A\big)^{-1} \big\|_{\text{op}}   \cdot \Big(  \big\|\big((1+\beta^2)I-\beta A\big) \cov(\mu) - I \big\|_{\text{op}}  + 1 \Big)  \\
&\le \eps^{-1}  \Big(  \big\|\big((1+\beta^2)I-\beta A\big) \cov(\mu) - I \big\|_{F}  + 1 \Big) \, .
\end{align}
Using Theorem~\ref{thm:main2},
\begin{align}
\E\big[\|\cov(\mu)\|_{\text{op}} \,  \mathbf{1}_{\cE^c}  \big]  &\le  \eps^{-1}  \Big( \E\big[\big\|\big((1+\beta^2)I-\beta A\big) \cov(\mu) - I \big\|_{F}\big]  + 1 \Big) \\
&\le \eps^{-1} C(\beta)\, .
 \end{align}
Next, on the event $\cE$, we trivially upper bound $\|\cov(\mu)\|_{\text{op}}$ by $n$, allowing us to conclude:
\begin{align}
\E\big[\|\cov(\mu)\|_{\text{op}} \big] &= \E\big[\|\cov(\mu)\|_{\text{op}}  \, \mathbf{1}_{\cE} \big] + \E\big[\|\cov(\mu)\|_{\text{op}}  \,  \mathbf{1}_{\cE^c} \big]\\
&\leq n\exp(-c(\eps)n) + \eps^{-1} C(\beta) \label{eq:bound12}\\
&\leq  (e c(\eps))^{-1} + \eps^{-1} C(\beta) \, .
\end{align}
\end{proof}

\section{Nonzero external field and related work} 
\label{sec:2}
The result of Theorem~\ref{thm:main2} is inspired by the following heuristic. Let us introduce an external field $y = (y_i)_{i=1}^n \in \R^n$ to the Gibbs measure:
\begin{equation}
\mu_{y}(\sigma) = \, \frac{1}{Z(y)} \exp\Big\{ \frac{\beta}{\sqrt{n}}\sum_{ i<j } g_{ij} \sigma_i \sigma_j + \sum_{i=1}^n y_i \sigma_i \Big\}\, .\end{equation}    
For a ``typical" external field, the log-partition function of $\mu_{y}$ is expected to have a TAP representation of the following form:  
\begin{align}\label{eq:logZ}
\log Z(y) ~&= \max_{m \in (-1,1)^n} \, \Big\{\cuF_{\sTAP}(m) + \langle y , m \rangle\Big\} + o_{\P}(n) \, , ~~~~~ \mbox{where}\\
\cuF_{\sTAP}(m) &:= \frac{\beta}{2}  \langle m, A m \rangle  + \sum_{i=1}^n h(m_i) + \frac{n \beta^2}{4} (1-Q(m))^2 \, ,\label{eq:TAP}\\
Q(m) &= \frac{1}{n} \|m\|^2 \, ,~~~ h(m) = -\frac{1+m}{2}\log \left(\frac{1+m}{2}\right) - \frac{1-m}{2}\log \left(\frac{1-m}{2}\right)  \, .\nonumber
\end{align}
The above representation was recently proved for Gaussian external fields by Chen, Panchenko and Subag~\cite{chen2018generalized}. Taking two derivatives with respect to $y$ on both sides of Eq.~\eqref{eq:logZ}, we expect 
\begin{align}
\cov(\mu_{y}) &:= \Big(\langle \sigma_i \sigma_j\rangle  -  \langle \sigma_i \rangle  \langle \sigma_j\rangle\Big)_{i,j=1}^n \approx  -\big(\nabla^2 \cuF_{\sTAP}(m)\big)^{-1}  \\
&= \Big(D(m) - \beta A + \beta^2\big(1- Q(m)\big) I - (2\beta^2/n) m m^\top\Big)^{-1} \, ,\label{eq:nonzeroy}
\end{align}
where $m$ is a maximizer in~\eqref{eq:logZ} and $D(m)$ is the diagonal matrix with entries $D(m)_{ii} = 1/(1-m_i^2)$.     
Taking $y=0$, we have $m=0$ and 
\begin{equation}\label{eq:heur} 
\cov(\mu) \approx \Big((1+\beta^2) I - \beta A\Big)^{-1} \, .
\end{equation}
The approximate resolvent identity~\eqref{eq:nonzeroy} was also heuristically derived by~\cite[Eq.\ (1.12)]{adhikari2021dynamical} by directly differentiating the TAP equations for the magnetizations in the external fields, and in fact, by much earlier work~\cite[Eq.\ (3.3)]{plefka1982convergence} by equating the covariance matrix with the second-order terms in a power series expansion of the Gibbs potential.

Theorem~\ref{thm:main2} makes the above approximation precise in the sense of a bounded Frobenius norm, once the inverse is eliminated from the right-hand side. This operation will allow us to perform Gaussian integration by parts with respect to the disorder random variables $g_{ij}$ and this creates various overlap terms between independent replicas from the Gibbs measure $\mu$. Then Theorem~\ref{thm:main2} is proved by exploiting known asymptotics for these overlaps in the high-temperature regime.        

Similarly to the zero external field case, one can attempt to make the approximation~\eqref{eq:nonzeroy} precise by showing that for $\beta$ small enough, e.g., below the AT line when $y$ is Gaussian, we have  
\begin{equation} 
\E \Big[\Big\| \Big( D(m) - \beta A + \beta^2(1- Q(m))I - (2\beta^2/n) m m^\top \Big) \cov(\mu_{y}) - I \Big\|_{F}^2\Big] \le C(\beta) \, ,
\end{equation}
where $m = (\langle \sigma_i \rangle)_{i=1}^n$ is the mean vector of $\mu_{y}$. 

Boundedness of $\E \|\cov(\mu_{y})\|_{\text{op}}$ would then follow if one can show that $\lambda_{\text{max}}\big(\nabla^2 \cuF_{\sTAP}(m)\big) \le -\eps$ for some $\eps>0$ with probability at least $1- O(1/n)$ (see Eq.~\eqref{eq:bound12}). Celentano~\cite{celentano2022sudakov} recently showed that the TAP free energy $\cuF_{\sTAP}$ is locally strongly concave around one of its stationary points with probability $1-o_n(1)$ by applying a Gaussian comparison theorem carefully conditioned on a sequence of sigma-fields produced by an Approximate Message Passing (AMP) iteration. This stationary point should presumably be close to the mean vector $\langle \sigma \rangle$; see the related works~\cite{chen2021convergence,celentano2021local}. The standard theory of AMP used in that paper does not yield any quantitative control on the probability of convergence, and this approach seems to fall short of obtaining the $O(1/n)$ rate needed to obtain operator norm bounds in expectation.       

Following the technique developed in Eldan and Shamir~\cite{eldan2022log}, which was later generalized in Chen and Eldan~\cite{chen2022localization}, a bound on operator norm of $\cov(\mu_{y})$ for an external field $y$ given by Eldan's stochastic localization process can be used to prove Poincar\'e and log-Sobolev inequalities for $\mu$. Our result can be seen as a small step within this larger scope. 

We finally mention that a few weeks after a version of this manuscript was made public, a concurrent work~\cite{brennecke2022two} established~\eqref{eq:heur} in the spectral sense: $\|\cov(\mu)-((1+\beta^2) I - \beta A)^{-1}\|_{\text{op}}\to 0$ in probability for $\beta<1$ using the cluster expansion technique. Their arguments are completely independent of ours. Subsequent work~\cite{brennecke2023operator} established boundedness in probability of $\|\cov(\mu_{h\cdot \mathbf{1}})\|_{\text{op}}$ for any fixed $h\geq 0$ in a suitable region below the AT line by combining the approach outlined above of utilizing overlap moments and the AMP iteration, with the dynamical techniques of~\cite{adhikari2021dynamical}.

\section{Proof of Theorem~\ref{thm:main2}}
\label{sec:proof}
We will show that the left-hand side in Eq.~\eqref{eq:main1} is actually equal to 
\begin{equation} \label{eq:beta4} 
\frac{2\beta^4}{(1-\beta^2)^2} + O(1/\sqrt{n}) \, .
\end{equation}

\begin{remark}\label{rem:beta4}
As an aside we note that for $\beta$ small, the above quantity behaves like $O(\beta^4)$, which is in agreement with the prediction made in~\cite[Section 11.5]{talagrand2011mean2} that $\cov(\mu) \simeq I+ \beta A + o(\beta)$, since $((1+\beta^2)I - \beta A)(I+\beta A)- I = O(\beta^2)$.
\end{remark}

We let
\[P = \big((1+\beta^2) I - \beta A\big)\cov(\mu) \, .\]
We then write 
\[\big\| P - I\big\|_{F}^2 = \big\| P \big\|_{F}^2 - 2 \Tr(P) + n \, .\]
 We treat the above expression term by term and we show that both $\E \| P \|_{F}^2$ and $\E\Tr(P)$ are of the form  
 \[n + C(\beta) + O(1/\sqrt{n})\, , \]  
 therefore canceling the terms diverging in $n$.  We first analyze the trace term and turn to the norm term which is more delicate. 
 
For notation, we write $\partial_{k\ell}:=\frac{\rmd}{\rmd g_{k\ell}}$. Let us first record the following simple lemma for future reference:   

\begin{lemma}
\label{lem:partials}
 For all $i,j,k,\ell$, it holds that $\partial_{k\ell} \langle \sigma_i\sigma_j\rangle=\frac{\beta}{\sqrt{n}}\left[\langle \sigma_i\sigma_j\sigma_k\sigma_{\ell}\rangle - \langle \sigma_i\sigma_j\rangle\langle \sigma_k\sigma_{\ell}\rangle\right]$.
\end{lemma}
\begin{proof}
The above is trivially true for $k=\ell$ since the Hamiltonian has no dependence on $g_{kk}$, and $\sigma_k^2=1$. Assume $k \neq \ell$. Writing $\langle f\rangle$ as a quotient, the product rule implies that the first term can be interpreted as the Gibbs average of $f\sigma_k\sigma_{\ell}$, while the second term obtained by differentiating the partition function yields an independent Gibbs average of $\sigma_k\sigma_{\ell}$.
\end{proof}

For $\ell$ independent replicas $\sigma^1,\cdots,\sigma^\ell$ drawn from $\mu$, we write $R_{a,b} = \frac{1}{n} \sum_{i=1}^n \sigma_i^a \sigma_i^b$ for the pairwise overlap between $\sigma^a$ and $\sigma^b$, and 
\[R_{1,\cdots,\ell} = \frac{1}{n} \sum_{i=1}^n \prod_{a=1}^\ell \sigma_i^a \]
for the multi-overlap of the replicas $\sigma^1,\cdots,\sigma^{\ell}$. 

\paragraph{The trace term:} Using Gaussian integration by parts and Lemma~\ref{lem:partials}, we have  
\begin{align*}
\E\Tr(P)&=(1+\beta^2)n-\beta \E\Tr(A\cov(\mu)) =(1+\beta^2)n- \frac{\beta}{\sqrt{n}}\sum_{1 \le i \neq j \le n} \E\big[g_{ij}\langle \sigma_i\sigma_j\rangle\big] \\
&=(1+\beta^2)n -  \frac{\beta^2}{n}\sum_{i,j=1}^{n} \big(1-\E \langle \sigma_i\sigma_j\rangle^2\big) \\
&= n + n \beta^2\E \big\langle R_{12}^2\big\rangle  \, .
\end{align*}  
From~\cite[Theorem 11.5.4]{talagrand2011mean2}, it known that for $\beta<1$,
\begin{equation*}
 n\E\big\langle R_{12}^2\big\rangle = \frac{1}{1-\beta^2}+O(1/\sqrt{n}) \, .
\end{equation*}
Therefore
\begin{equation}
\label{eq:trace}
\E \Tr(P) =n+\frac{\beta^2}{1-\beta^2}+O(1/\sqrt{n})\, .
\end{equation}

\paragraph{The norm term:} We now calculate $\E\| P \|_{F}^2$.  We have
\begin{align*}
\E \|P\|_F^2&=\E \sum_{i,j=1}^n \Big((1+\beta^2)\langle \sigma_i\sigma_j\rangle -\frac{\beta}{\sqrt{n}}\sum_{k\neq i} g_{ik}\langle \sigma_j\sigma_k\rangle\Big)^2\\
&=\E\sum_{i,j=1}^n \bigg((1+\beta^2)^2\langle \sigma_i\sigma_j\rangle^2 - 2(1+\beta^2)\frac{\beta}{\sqrt{n}}\sum_{k\neq i} g_{ik}\langle \sigma_j\sigma_k\rangle\langle \sigma_i\sigma_j\rangle\\
&~~~~~~~~~~~~+\frac{\beta^2}{n} \sum_{k ,\ell \neq i} g_{ik}g_{i\ell}\langle \sigma_j\sigma_k\rangle\langle \sigma_j\sigma_{\ell}\rangle\bigg) \\
&= \text{I} + \text{II} + \text{III}\, .
\end{align*}

We deal with these terms in order. Since
\begin{equation*}
\sum_{i,j=1}^n \langle \sigma_i\sigma_j\rangle^2=\sum_{i,j=1}^n \langle \sigma_i^1\sigma_j^1\sigma_i^2\sigma_j^2\rangle=n^2\langle R_{12}^2\rangle \, ,
\end{equation*}
the first term is 
\begin{equation}
\label{eq:term1}
\text{I} = (1+\beta^2)^2 n^2\E\big\langle R_{12}^2\big\rangle \, .
\end{equation}

As for the second term, using Lemma~\ref{lem:partials} we have
\begin{align*}
\frac{\beta}{\sqrt{n}} \E\big[g_{ik}\langle \sigma_j\sigma_k\rangle\langle\sigma_i\sigma_j\rangle\big] &=
\frac{\beta}{\sqrt{n}} \E\left[\partial_{ik}[\langle \sigma_j\sigma_k\rangle\langle\sigma_i\sigma_j\rangle]\right] \\
&=\frac{\beta^2}{n} \E\big[(\langle \sigma_i\sigma_j\rangle - \langle\sigma_i\sigma_k\rangle\langle \sigma_j\sigma_k\rangle)\langle \sigma_i\sigma_j\rangle + (\langle \sigma_j\sigma_k\rangle-\langle \sigma_i\sigma_k\rangle\langle\sigma_i\sigma_j\rangle)\langle\sigma_j\sigma_k\rangle\big]\\
 &=\frac{\beta^2}{n}\E\big[\langle \sigma_i\sigma_j\rangle^2+\langle \sigma_j\sigma_k\rangle^2- 2\langle \sigma_i\sigma_j\rangle\langle \sigma_j\sigma_k\rangle\langle\sigma_i\sigma_k\rangle\big] \, .
\end{align*}

Summing the above expression over $i,j,k$, (note that the above is $0$ for $k=i$) this is 
\begin{equation*}
  2\beta n^2 \E\big\langle R_{12}^2\big\rangle + 2\beta n^2 \E\big\langle R_{13}R_{12}R_{23}\big\rangle \, .
\end{equation*}
Therefore 
\begin{equation}
\label{eq:term2}
\text{II} = -4\beta^2(1+\beta^2)n^2\left(\E\big\langle R_{12}^2\big\rangle - \E\big\langle R_{12}R_{23}R_{13}\big\rangle \right) \, .
\end{equation}

We now turn to the third term. We split the sum into a diagonal and a off-diagonal part:
\begin{equation}\label{eq:split}
\sum_{k,\ell \in [n] \setminus\{i\}}  \langle \sigma_j\sigma_k\rangle\langle \sigma_j\sigma_{\ell}\rangle = \sum_{k\neq i} g_{ik}^2\langle \sigma_j\sigma_k\rangle^2 + \sum_{k, \ell \, \in [n] \setminus\{i\}, k\neq \ell} g_{ik}g_{i\ell}\langle \sigma_j\sigma_k\rangle\langle \sigma_j\sigma_{\ell}\rangle \, .
\end{equation}
For the diagonal term in Eq.~\eqref{eq:split}, taking expectations and applying Gaussian integration by parts, we find
\begin{align*}
\frac{1}{n}\sum_{k \neq i} \E[g_{ik}^2\langle \sigma_j\sigma_k\rangle^2] 
&=\frac{1}{n}\sum_{k \neq i} \E[\langle \sigma_j\sigma_k\rangle^2]+\frac{2\beta}{n^{3/2}}\sum_{k=1}^n \E[g_{ik}\langle \sigma_j\sigma_k\rangle(\langle \sigma_i\sigma_j\rangle-\langle \sigma_i\sigma_k\rangle\langle\sigma_j\sigma_k\rangle)]\\
&=\frac{1}{n}\sum_{k\neq i} \E[\langle \sigma_j\sigma_k\rangle^2]\\
&~~~+\frac{2\beta^2}{n^{2}}\sum_{k=1}^n\E[(\langle \sigma_i\sigma_j\rangle-\langle \sigma_i\sigma_k\rangle\langle\sigma_j\sigma_k\rangle)^2]\\
&~~~+\frac{2\beta^2}{n^{2}}\sum_{k=1}^n\E[\langle \sigma_j\sigma_k\rangle(\langle \sigma_j\sigma_k\rangle-\langle \sigma_i\sigma_j\rangle\langle \sigma_i\sigma_k\rangle)]\\
&~~~-\frac{2\beta^2}{n^{2}}\sum_{k=1}^n\E[\langle \sigma_j\sigma_k\rangle (1-\langle \sigma_i\sigma_k\rangle^2)\langle \sigma_j\sigma_k\rangle]\\
&~~~-\frac{2\beta^2}{n^{2}}\sum_{k=1}^n\E[\langle \sigma_j\sigma_k\rangle \langle \sigma_i\sigma_k\rangle (\langle \sigma_i\sigma_j\rangle -\langle \sigma_i\sigma_k\rangle\langle \sigma_j\sigma_k\rangle)]\\
&=\frac{1}{n}\sum_{k\neq i}\E[\langle \sigma_j\sigma_k\rangle^2] \\
&~~+\frac{2\beta^2}{n^2}\sum_{k=1}^n \E\left[\langle \sigma_i\sigma_j\rangle^2-4\langle \sigma_i\sigma_j\rangle\langle \sigma_i\sigma_k\rangle\langle \sigma_j\sigma_k\rangle+3\langle \sigma_i\sigma_k\rangle^2\langle \sigma_j\sigma_k\rangle^2\right].
\end{align*}
Summing over $i,j$ we find  
\begin{align}
\label{eq:term3}
\frac{\beta^2}{n} \sum_{1 \le i,j,k \le n, i \neq k} \E\big[g_{ik}^2 \langle \sigma_j\sigma_k\rangle^2\big] 
&= \beta^2n(n-1) \E\big\langle R_{12}^2\big\rangle \\
&+ 2\beta^4 n \Big(\E\big\langle R_{12}^2\big\rangle - 4\E\big\langle R_{12}R_{13}R_{23}\big\rangle] + 3\E\big\langle R_{12}R_{34}R_{1234}\big\rangle\Big).\nonumber
\end{align}

For the off-diagonal term in Eq.~\eqref{eq:split}, since the random variables $g_{ij}$ are independent, we have
\begin{align*}
\frac{1}{n}\E[g_{ik}g_{i\ell}\langle \sigma_j\sigma_k\rangle \langle \sigma_j\sigma_{\ell}\rangle]
&=\frac{\beta}{n^{3/2}}\E[g_{i\ell}\left((\langle \sigma_i\sigma_j\rangle-\langle \sigma_j\sigma_k\rangle\langle \sigma_i\sigma_k\rangle) \langle \sigma_j\sigma_{\ell}\rangle +\langle \sigma_j\sigma_k\rangle (\langle\sigma_i\sigma_k\sigma_j\sigma_{\ell}\rangle-\langle \sigma_i\sigma_k\rangle\langle \sigma_j\sigma_{\ell}\rangle)\right)]\\
&= \frac{\beta^2}{n^2}\E[(\langle \sigma_j\sigma_{\ell}\rangle -\langle \sigma_i\sigma_{\ell}\rangle \langle \sigma_i\sigma_j\rangle)\langle \sigma_j\sigma_{\ell}\rangle]\\
&~~~-\frac{\beta^2}{n^2}\E[(\langle \sigma_i\sigma_j\sigma_k\sigma_{\ell}\rangle -\langle \sigma_i\sigma_{\ell}\rangle \langle \sigma_j\sigma_k\rangle)\langle \sigma_i\sigma_k\rangle\langle \sigma_j\sigma_{\ell}\rangle]\\
&~~~-\frac{\beta^2}{n^2}\E[\langle \sigma_j\sigma_k\rangle(\langle \sigma_k\sigma_{\ell}\rangle -\langle \sigma_i\sigma_{\ell}\rangle \langle \sigma_i\sigma_{k}\rangle)\langle \sigma_j\sigma_{\ell}\rangle]\\
&~~~+\frac{\beta^2}{n^2}\E[(\langle \sigma_i\sigma_j\rangle-\langle \sigma_j\sigma_k\rangle\langle \sigma_i\sigma_k\rangle)(\langle\sigma_i\sigma_j\rangle-\langle \sigma_i\sigma_{\ell}\rangle\langle \sigma_j\sigma_{\ell}\rangle)]\\
&~~~+\frac{\beta^2}{n^2}\E[(\langle \sigma_i\sigma_j\sigma_k\sigma_{\ell}\rangle-\langle \sigma_j\sigma_k\rangle\langle \sigma_i\sigma_{\ell}\rangle)(\langle\sigma_i\sigma_k\sigma_j\sigma_{\ell}\rangle-\langle \sigma_i\sigma_k\rangle\langle \sigma_j\sigma_{\ell}\rangle)]\\
&~~~+\frac{\beta^2}{n^2}\E[\langle \sigma_j\sigma_k\rangle (\langle \sigma_j\sigma_k\rangle -\langle \sigma_i\sigma_j\sigma_k\sigma_{\ell}\rangle \langle \sigma_i\sigma_{\ell}\rangle)]\\
&~~~-\frac{\beta^2}{n^2}\E[\langle \sigma_j\sigma_k\rangle (\langle \sigma_k\sigma_{\ell}\rangle -\langle \sigma_i\sigma_{k}\rangle \langle \sigma_i\sigma_{\ell}\rangle)\langle \sigma_j\sigma_{\ell}\rangle]\\
&~~~-\frac{\beta^2}{n^2}\E[\langle \sigma_j\sigma_k\rangle \langle \sigma_i\sigma_k\rangle (\langle \sigma_i\sigma_j\rangle -\langle \sigma_j\sigma_{\ell}\rangle \langle \sigma_i\sigma_{\ell}\rangle)] \, .
\end{align*}

Let's call the above expression $O_{ijk\ell}$. Summing this expression over \emph{all} $i,j,k,\ell$ and only subtracting off the diagonal terms corresponding to $k=\ell$ later (note again that the terms $k=i$ or $\ell=i$ vanish) gives

\begin{align*}
\sum_{i,j,k,\ell}O_{ijk\ell} &= \frac{\beta^2}{n^2}\E[n^4\langle R_{12}^2\rangle - n^4 \langle R_{12}R_{13}R_{23}\rangle] \\
&~~~-\frac{\beta^2}{n^2}\E[n^4\langle R_{12}^2R_{13}^2\rangle -n^4\langle R_{13}R_{24}R_{23}R_{14}\rangle]\\
&~~~-\frac{\beta^2}{n^2}\E[n^4\langle R_{12}R_{13}R_{23}\rangle-n^4\langle R_{13}R_{14}R_{24}R_{23}\rangle]\\
&~~~+\frac{\beta^2}{n^2}\E[n^4\langle R_{12}^2\rangle-2n^4\langle R_{12}R_{13}R_{23}\rangle+n^4\langle R_{14}R_{13}R_{23}R_{24}\rangle]\\
&~~~+\frac{\beta^2}{n^2}\E[n^4\langle R_{12}^4\rangle-2n^4\langle R_{12}^2R_{13}^2\rangle+n^4\langle R_{14}R_{13}R_{23}R_{24}\rangle]\\
&~~~+\frac{\beta^2}{n^2}\E[n^4\langle R_{12}^2\rangle-n^4\langle R_{12}^2R_{23}^2\rangle]\\
&~~~-\frac{\beta^2}{n^2}\E[n^4\langle R_{12}R_{13}R_{23}\rangle -n^4\langle R_{14}R_{12}R_{23}R_{34}\rangle]\\
&~~~-\frac{\beta^2}{n^2}\E[n^4\langle R_{12}R_{13}R_{23}\rangle -n^4\langle R_{12}R_{13}R_{24}R_{34}\rangle] \, .
\end{align*}

Collecting terms and multiplying by $\beta^2$, we get
\begin{equation}
\label{eq:term4}
\beta^2 \sum_{i,j,k,\ell}O_{ijk\ell}  = \beta^4 n^2\E\big[3\langle R_{12}^2\rangle+\langle R_{12}^4\rangle-6\langle R_{12}R_{13}R_{23}\rangle-4\langle R_{12}^2R_{23}^2\rangle+6\langle R_{12}R_{23}R_{34}R_{14}\rangle\big] \, .
\end{equation}

Now considering the diagonals of that expression where $k=\ell$, and summing over $i,j,k$, and multiplying by $\beta^2$ again, we get

\begin{align*}
 \beta^2 \sum_{i,j,k}O_{ijkk} &=  \frac{\beta^4}{n^2}\E[n^3\langle R_{12}^2\rangle -n^3\langle R_{12}R_{13}R_{23}\rangle]\\
 &~~~-\frac{\beta^4}{n^2}\E[n^3\langle R_{12}R_{23}R_{13}\rangle -n^3 \langle R_{13}R_{24}R_{1234}\rangle]\\
&~~~-\frac{\beta^4}{n^2}\E[n^3\langle R_{12}^2\rangle-n^3\langle R_{13}R_{24}R_{1234}\rangle]\\
&~~~+\frac{\beta^4}{n^2}\E[n^3\langle R_{12}^2\rangle-2n^3\langle R_{12}R_{23}R_{13}\rangle+n^3\langle R_{13}R_{24}R_{1234}\rangle]\\
&~~~+\frac{\beta^4}{n^2}\E[n^3\langle R_{12}^2\rangle-2n^3\langle R_{12}R_{23}R_{13}\rangle+n^3\langle R_{13}R_{24}R_{1234}\rangle]\\
&~~~+\frac{\beta^4}{n^2}\E[n^3\langle R_{12}^2\rangle -n^3\langle R_{12}R_{23}R_{13}\rangle]\\
&~~~-\frac{\beta^4}{n^2}\E[n^3\langle R_{12}^2\rangle -n^3\langle R_{13}R_{24}R_{1234}\rangle]\\
&~~~-\frac{\beta^4}{n^2}\E[n^3\langle R_{12}R_{23}R_{13}\rangle -n^3\langle R_{13}R_{24}R_{1234}\rangle] \, .
\end{align*}
After simplification, 
\begin{equation}
\label{eq:term5}
 \beta^2 \sum_{i,j,k}O_{ijkk}  = \beta^4 n \E\big[2\langle R_{12}^2\rangle-8\langle R_{12}R_{13}R_{23}\rangle+6\langle R_{13}R_{24}R_{1234}\rangle\big] \, .
\end{equation}

Putting together \eqref{eq:term3}, \eqref{eq:term4} and \eqref{eq:term5}, we obtain
\begin{align}\label{eq:term6}
\text{III} &= \beta^2n(n-1) \E[\langle R_{12}^2\rangle] + 2\beta^4 n \left(\E[\langle R_{12}^2\rangle]-4\E[\langle R_{12}R_{13}R_{23}\rangle] + 3\E[\langle R_{12}R_{34}R_{1234}\rangle]\right) \nonumber\\
&~~~+\beta^4 n^2\E[3\langle R_{12}^2\rangle+\langle R_{12}^4\rangle-6\langle R_{12}R_{13}R_{23}\rangle-4\langle R_{12}^2R_{23}^2\rangle+6\langle R_{12}R_{23}R_{34}R_{14}\rangle]\nonumber\\
&~~~-\left(\beta^4 n \E[2\langle R_{12}^2\rangle-8\langle R_{12}R_{13}R_{23}\rangle+6\langle R_{13}R_{24}R_{1234}\rangle]\right) \nonumber\\
&=\beta^2n(n-1) \E[\langle R_{12}^2\rangle] \nonumber\\
&~~~ +\beta^4 n^2\E[3\langle R_{12}^2\rangle+\langle R_{12}^4\rangle-6\langle R_{12}R_{13}R_{23}\rangle-4\langle R_{12}^2R_{23}^2\rangle+6\langle R_{12}R_{23}R_{34}R_{14}\rangle] \, .
\end{align}
Combining~\eqref{eq:term1},~\eqref{eq:term2} and~\eqref{eq:term6},  we obtain an expression for $\E \|P\|_F^2$:
\begin{align}\label{eq:normP}
\E \|P\|_F^2&=(1+\beta^2)^2n^2\E[\langle R_{12}^2\rangle]
-4\beta^2(1+\beta^2)n^2\left(\E[\langle R_{12}^2\rangle]-\E[\langle R_{12}R_{13}R_{23}\rangle]\right)\nonumber\\
&~~~+\beta^2n(n-1) \E[\langle R_{12}^2\rangle] \nonumber\\
&~~~+\beta^4 n^2\E[3\langle R_{12}^2\rangle+\langle R_{12}^4\rangle-6\langle R_{12}R_{13}R_{23}\rangle-4\langle R_{12}^2R_{23}^2\rangle+6\langle R_{12}R_{23}R_{34}R_{14}\rangle]\nonumber\\
&=(1-\beta^2)n^2 \E[\langle R_{12}^2\rangle]- \beta^2 n \E[\langle R_{12}^2\rangle] +\left(4\beta^2(1+\beta^2)-6\beta^4\right)n^2\E[\langle R_{12}R_{13}R_{23}\rangle]\nonumber\\
&~~~+\beta^4 n^2\E[\langle R_{12}^4\rangle] -4\beta^4n^2\E[\langle R_{12}^2R_{23}^2\rangle]+6\beta^4n^2\E[\langle R_{12}R_{23}R_{34}R_{14}\rangle] \, .
\end{align}

From Talagrand~\cite[Theorem 11.5.4]{talagrand2011mean2} and Bardina, M\'arquez-Carreras, Rovira and Tindel ~\cite{bardina2004higher}, we have the precise asymptotic expansions up to order $O(n^{-5/2})$ of all the overlaps involved in the above expressions:
\begin{gather}
\E\big\langle R_{12}^2\big\rangle = \frac{1}{n(1-\beta^2)}+\frac{-\beta^2(1+\beta^2)}{n^2(1-\beta^2)^4}+O(1/n^{5/2}) \, ,\\
\E\big\langle R_{12}R_{13}R_{23}\big\rangle= \frac{1}{n^2(1-\beta^2)^3}+O(1/n^{5/2}) \, ,\\
\E\big\langle R_{12}^4\big\rangle = \frac{3}{n^2(1-\beta^2)^2}+O(1/n^{5/2}) \, ,\\
\E\big\langle R_{12}^2R_{23}^2\big\rangle = \frac{1}{n^2(1-\beta^2)^2}+O(1/n^{5/2}) \, \\
\E\big\langle R_{12}R_{23}R_{34}R_{14}\big\rangle = O(1/n^{5/2}) \, .
\end{gather}
From this we deduce
\begin{align}\label{eq:normP2}
\E \|P\|_F^2&= n - \frac{\beta^2}{1-\beta^2} - \frac{\beta^2(1+\beta^2)}{(1-\beta^2)^3} + \frac{4\beta^2(1+\beta^2)-6\beta^4}{(1-\beta^2)^3} + \frac{3\beta^4}{(1-\beta^2)^2}-\frac{4\beta^4}{(1-\beta^2)^2} + O(1/\sqrt{n}) \\ 
&= n+ \frac{2\beta^2}{(1-\beta^2)^2}  + O(1/\sqrt{n}) \, .
\end{align}

Finally, combining the above formula with Eq.~\eqref{eq:trace} for $\E \Tr(P)$ we obtain
\begin{equation}
\E\big[\big\|P-I\big\|_F^2\big] =\frac{2\beta^4}{(1-\beta^2)^2}  + O(1/\sqrt{n}) \, .
\end{equation}

\section{Lower bounds at the critical temperature and low temperatures}
\label{sec:lowerb}
In this section, we investigate the tightness of our results. In particular, we show that the expected operator norm necessarily diverges as $\beta\to 1$. In the low temperature regime where $\beta>1$, we provide a simple linear lower bound on the expected operator norm. Finally, we consider the behavior of the operator norm near and at the critical temperature $\beta = 1$.

Below, we write $\cov(\mu_{\beta,n})$ to emphasize the dependence on the inverse temperature $\beta$ and $n$. Our results are derived using available results on moment overlaps combined with the following elementary claim:

\begin{lemma}
\label{lem:op_lb}
    For any $\beta \ge 0$, it holds that
    \begin{equation*}
        \E\big[\|\mathrm{cov}(\mu_{\beta})\|_{\mathrm{op}}\big]\geq n\E\langle R_{12}^2\rangle.
    \end{equation*}
\end{lemma}
\begin{proof}
 By the variational formula for eigenvalues, it holds deterministically that for any $\sigma\in \{-1,1\}^n$,
 \begin{equation*}
     \|\mathrm{cov}(\mu_{\beta})\|_{\mathrm{op}}\geq \frac{1}{n} \sigma^\top\mathrm{cov}(\mu_{\beta})\sigma=\frac{1}{n} \sigma^\top\langle \sigma_1\sigma_1^\top\rangle\sigma.
 \end{equation*}
 Averaging over $\sigma \sim \mu_{\beta}$ and then taking expectations over the disorder yields the claim.
\end{proof}

\subsection{High, near critical temperature}
A simple consequence of~\ref{lem:op_lb} is that the expected operator norm necessarily diverges as $\beta\to 1^-$:

\begin{proposition}
\label{prop:hightemp}
For $\beta<1$, it holds that 
\begin{equation}
\liminf_{n\to \infty} \E \big\|\cov(\mu_{\beta,n})\big\|_{\textup{op}} \geq \frac{1}{1-\beta^2} \, .
\end{equation}
In particular, 
\begin{equation}
\lim_{\beta\to 1^-}\liminf_{n\to \infty} \E \big\|\cov(\mu_{\beta,n}) \big\|_{\textup{op}} = + \infty \, .
\end{equation}
\end{proposition}
\begin{proof}
From Lemma~\ref{lem:op_lb} and using Talagrand~\cite[Theorem 11.5.4]{talagrand2011mean2}, we immediately obtain
\begin{equation*}
    \mathbb{E}[\|\mathrm{cov}(\mu_{\beta})\|_{\mathrm{op}}]\geq n\mathbb{E}[\langle R_{12}^2\rangle]\geq \frac{1}{1-\beta^2}-O(1/n).
\end{equation*}
\end{proof}

Note that the previous result does not appear to immediately have any bearing on the behavior at $\beta = 1$ due to the subtlety of interchanging limits. We treat the case $\beta=1$ below using a similar analysis and leveraging some of the few known results in this regime.

\subsection{Low temperature}
Next, we show that the boundedness of the operator norm cannot extend past $\beta=1$; in fact, the operator norm necessarily grows linearly in $n$.

\begin{proposition}
For every $\beta>1$, there exists a constant $c(\beta)>0$ such that 
\begin{equation}
\E \big\|\cov(\mu_{\beta,n})\big\|_{\textup{op}}\geq c(\beta)\, n \, .
\end{equation}
\end{proposition}
\begin{proof}
From Talagrand~\cite[Equation (14.417)]{talagrand2011mean2}, it holds that
\begin{equation}\label{eq:zeta}
\lim_{n\to\infty}\E \big\langle R_{12}^2 \big\rangle = \int q^2 \zeta^*_{\beta}(\rmd q) \, , 
\end{equation}
where $\zeta^*_{\beta}$ is called the Parisi measure and is the unique minimizer of the Parisi functional; see~\cite{auffinger2015parisi} and~\cite{talagrand2011mean2,panchenko2013sherrington} for definitions. It is further known that $\zeta^*_{\beta} \neq \delta_{0}$ when $\beta>1$, so the right-hand side in~\eqref{eq:zeta} is some constant $c(\beta)>0$; see~\cite[Section 5.2]{alaoui2022sampling} for more details. The previous display with Lemma~\ref{lem:op_lb} yields the desired lower bound.
\end{proof}

\subsection{At the transition}
To complete this picture, we now consider the operator norm near and at the critical temperature $\beta=1$. The next result refines Proposition~\ref{prop:hightemp} by showing that for a sequence $\beta_n\to 1^-$, one can obtain similar lower bounds on the operator norm with different exponents:

\begin{proposition}\label{prop:near_crit}
Fix any nonnegative function $f(n)$ satisfying $f(n)\to \infty$ and $f(n)=o(n^{1/3})$. There exists a sequence $\beta_n\to 1^-$ such that for all but finitely many $n$,
\begin{equation}
\E \big\|\cov(\mu_{\beta_n,n})\big\|_{\textup{op}} \ge f(n) \, .
\end{equation}
\end{proposition}
\begin{proof}
For any such function $f(n)$, define a sequence $\beta_n$ by $f(n)=\frac{1}{2}\frac{1}{1-\beta_n^2}$. Then $n^{1/3}(1-\beta_n^2)\to\infty$ by construction. From Talagrand~\cite[Theorem 11.7.1]{talagrand2011mean2}, this condition implies 
\begin{equation*}
    n(1-\beta_n^2)\mathbb{E}[\langle R_{12}^2\rangle]\to 1\, .
\end{equation*}
This implies by Lemma~\ref{lem:op_lb} that for all large enough $n$, $\E \big\|\cov(\mu_{\beta_n,n})\big\|_{\textup{op}}\geq \frac{1}{2(1-\beta_n^2)}\geq f(n)$ as needed.
\end{proof}

Finally, we directly consider the behavior at the critical temperature $\beta=1$. While it is expected that $\E\langle R_{12}^2\rangle$ is increasing with respect to $\beta$, which would imply a lower bound from the previous results, this does not appear to be known. Instead, we give a slightly weaker polynomial lower bound that holds unconditionally by applying known results in this setting:

\begin{proposition}\label{prop:critical}
For $\beta=1$, there is an absolute constant $c>0$ such that
\begin{equation}
\E \big\|\cov(\mu_{\beta,n})\big\|_{\textup{op}} \ge c\Big(\frac{n}{\log n}\Big)^{3/16}  \, . 
\end{equation}
\end{proposition}
\begin{proof}
We leverage some known results on overlap convergence at $\beta = 1$. First, a result of Chatterjee~\cite[Theorem 11.7.6]{talagrand2011mean2} shows a lower bound on the third moment of $R_{12}$: there exists a constant $L >0$ such that
\begin{equation}
\E\big\langle \vert R_{12}\vert^3\big\rangle \geq \frac{1}{Ln} \, .
\end{equation}
Next, we need a bound on the decay of tail probability of $R_{12}$. This is the content of a result of Talagrand~\cite[Theorem 2.14.5]{talagrand2003challenge} which we quote here: for $\beta \le 1$ and for all $x \ge L (\log n / n)^{3/8}$ where $L = L(\beta) < \infty$, it holds that  
\begin{equation}\label{eq:tailR0}
\P\left(|R_{12}| \geq \sqrt{x} \right) \leq L n \exp\big( - n x^2 \big/ L\big) +  \exp\Big( - n^{3/2} x^4 \big/ (L\sqrt{\log n}) \Big)  \, ,
\end{equation}
where the probability is over the distribution of $R_{12}$ induced by the measure $\E\langle \, \cdot \, \rangle^{\otimes 2}$.
Taking $x = \max\{1, (2/L^3)^{1/4}\} \cdot L (\log n / n)^{3/8}$, we find that for $n$ large enough, 
\begin{equation}\label{eq:tailR}
\P\left(|R_{12}|\geq \sqrt{x} \right) \leq \frac{2}{n^2} \, .
\end{equation}
Combining these results shows that
\begin{equation*}
\frac{1}{Ln}\leq \E\big\langle \vert R_{12}\vert^3\mathbf{1}_{|R_{12}| \ge \sqrt{x}}\big\rangle + \E\big\langle \vert R_{12}\vert^3\mathbf{1}_{|R_{12}| < \sqrt{x}} \big\rangle \leq  \P\big(\vert R_{12}\vert \geq \sqrt{x}\big)+ \sqrt{x}\E\big\langle R_{12}^2\big\rangle \, .
\end{equation*}
 Using~\eqref{eq:tailR} and rearranging yields
\begin{equation}
\E\big\langle R_{12}^2\big\rangle\geq \frac{c}{n^{13/16}(\log n)^{3/16}} \, ,
\end{equation}
for some constant $c>0$. The result then follows immediately from Lemma~\ref{lem:op_lb}.
\end{proof}


\vspace{5mm}
\textbf{Acknowledgments.} The first author is grateful to Ronen Eldan for valuable conversations about bounding the covariance matrix of disordered spin systems which inspired this work.  We also thank the anonymous referees for their feedback, and for Remark~\ref{rem:beta4} on the predicted $O(\beta^4)$ behavior in the bound Eq.~\eqref{eq:beta4}. 
\vspace{5mm}

\bibliographystyle{amsalpha}
\bibliography{sn-bibliography}

\end{document}